\documentclass[11pt]{amsart}

\makeatletter
\usepackage{amssymb}
\usepackage{latexsym}
\usepackage{amsbsy}
\usepackage{amsfonts}

\def\marginpar#1{\ignorespaces}

\newtheorem{theorem}[equation]{Theorem}
\newtheorem{proposition}[equation]{Proposition}
\newtheorem{lemma}[equation]{Lemma}
\newtheorem{corollary}[equation]{Corollary}

\theoremstyle{definition}
\newtheorem{remark}[equation]{Remark}

\numberwithin{equation}{section}

\def\ff{\frac}
\def\rr{\rho}
\def\nn{\nabla}
\def\Hess{\rm{Hess}}
\def\AArm{\fam0 \rm}%
\newdimen\AAdi%
\newbox\AAbo%
\def\AAk#1#2{\setbox\AAbo=\hbox{#2}\AAdi=\wd\AAbo\kern#1\AAdi{}}%

\newcommand{\BBone}{{\ensuremath{{\AArm 1\AAk{-.8}{I}I}}}}

\def\eqref#1{(\ref{#1})}
\def\eqlabel#1{\def\@currentlabel{#1}}

\def\formula#1{\def\@tempa{#1}\let\@tempb\theequation\def\theequation{%
\hbox{#1}}\def\@currentlabel{(\theequation)}$$}
\def\endformula{\leqno\hbox{(\@tempa)}$$\@ignoretrue\let\theequation\@tempb}

\def\given{\hskip5\p@\relax\vrule\@width.4\p@\hskip5\p@\relax}

\newcommand{\open}[1]{%
\par\normalfont\topsep6\p@\@plus6\p@\trivlist\item[\hskip\labelsep\itshape#1%
\@addpunct{.}]\ignorespaces}

\DeclareRobustCommand{\close}[1]{%
  \ifmmode 
  \else \leavevmode\unskip\penalty9999 \hbox{}\nobreak\hfill
  \fi
  \quad\hbox{$#1$}}

\newlength{\toskip}

\def\FF{\mathcal F}

\def\EE{\mathcal E}

\def\LL{\mathcal L}

\def\XX{\mathcal X}

\def\<{\langle}
\def\>{\rangle}

\def \R {{\mathbb R}}

\def \P {{\mathbb P}}
\def \E {{\mathbb E}}
\def \N {{\mathbb N}}
\def \D {{\mathbb D}}
\def \L {{\mathbb L}}

\def \W {{\mathbb W}}

\def \Var {\textrm{Var}}
\def \Ent {\textrm{Ent}}
\makeatother

\begin{document}
\date{\today}

\title[Drift and Super poincar\'e]{Lyapunov conditions for logarithmic Sobolev and Super Poincar\'e inequality}

 \author[P. Cattiaux]{\textbf{\quad {Patrick} Cattiaux $^{\spadesuit}$ \, \, }}
\address{{\bf {Patrick} CATTIAUX},\\ Universit\'e Paul Sabatier
Institut de Math\'ematiques. Laboratoire de Statistique et Probabilit\'es, UMR C 5583\\ 118 route
de Narbonne, F-31062 Toulouse cedex 09.} \email{cattiaux@math.univ-toulouse.fr}

 \author[A. Guillin]{\textbf{\quad {Arnaud} Guillin $^{\diamondsuit}$}}
\address{{\bf {Arnaud} GUILLIN},\\ Ecole Centrale Marseille et LATP \, Universit\'e  de Provence,
Technopole Château-Gombert, 39, rue F. Joliot Curie, 13453 Marseille Cedex 13.} \email{aguillin@egim-mrs.fr}

\author[F.Y. Wang]{\textbf{\quad {Feng-Yu} Wang $^{*}$}}
\address{{\bf {Feng-Yu} WANG},\\ Department of Mathematics, Swansea University, Singleton Park, SA2 8PP, Swansea UK}
\email{wangfy@bnu.edu.cn}

\author[L. Wu]{\textbf{\quad {Liming} Wu $^{\heartsuit}$}}
\address{{\bf Liming WU},\\ Laboratoire de Math\'ematiques Appliqu\'ees, CNRS-UMR
6620, Universit\'e Blaise Pascal, 63177 Aubi\`ere, France. And Department of
Mathematics, Wuhan University, 430072 Hubei, China}
\email{Li-Ming.Wu@math.univ-bpclermont.fr}

\maketitle
 \begin{center}

 \textsc{$^{\spadesuit}$  Universit\'e de Toulouse}
\smallskip

\textsc{$^{\diamondsuit}$ Ecole Centrale Marseille \quad and \quad Universit\'e de Provence}
\smallskip

\textsc{$^{*}$ Swansea University}
\smallskip

\textsc{$^{\heartsuit}$ Universit\'e Blaise Pascal \quad and \quad Wuhan University}
 \end{center}

\begin{abstract}
We show how to use Lyapunov functions to obtain functional inequalities which are stronger than
Poincar\'e inequality (for instance logarithmic Sobolev or $F$-Sobolev). The case of Poincar\'e
and weak Poincar\'e inequalities was studied in \cite{BCG}. This approach allows us to recover and
extend in an unified way some known criteria in the euclidean case (Bakry-Emery, Wang, Kusuoka-Stroock ...).

\end{abstract}
\bigskip

\bigskip

\bigskip

\section{\bf Introduction.}\label{Intro}

During the last thirty years, a lot of attention has been devoted to the study of various
functional inequalities and among them a lot of efforts were consecrated to the logarithmic
Sobolev inequality. Our goal here will be to give a new and practical condition to prove
logarithmic Sobolev inequality in a general setting. Our method being general, we will be able to
get also conditions for Super-Poincar\'e, and by incidence to various inequalities as $F$-Sobolev
or general Beckner inequalities. Our assumptions are based mainly on a Lyapunov type condition as
well as a Nash inequality (for example valid in $\R^d$). But let us make precise the objects and
inequalities we are interested in.

Let $(\XX,\FF,\mu)$ be a probability space and $\LL$ a self adjoint operator on $L^2(\mu)$, with
domain $\D_2(\L)$, such that $P_t=e^{t\LL}$ is a Markov semigroup. Consider then the Dirichlet
form associated to $\LL$
$$\EE(f,f):=\langle -\LL f,f\rangle_\mu\qquad f\in\D_2(\LL)$$
with domain $\D(\EE)$. Throughout the paper, all test functions in an inequality will belong to
$\D(\LL)$.  It is well known that $\LL$ possesses a spectral gap if and only if the following
Poincar\'e inequality holds (for all nice $f$'s)
\begin{equation}\label{P}
\Var_\mu(f):=\int f^2d\mu-\left(\int fd\mu\right)^2\le C_P\,\EE(f,f)
\end{equation}
where $C_P^{-1}$ is the spectral gap. Note that such an inequality is also equivalent to  the
exponential decay in $L^2(\mu)$ of $P_t$.

A defective logarithmic Sobolev inequality (say DLSI) is satisfied if for all nice $f$'s
\begin{equation}\label{DLSI}
\Ent_\mu(f^2):=\int f^2\log f^2d\mu-\int f^2d\mu\log\left(\int f^2d\mu\right)\le C_{LS}\,\EE(f,f) +D_{LS}\int f^2d\mu.
\end{equation}
When $D_{LS}=0$ the inequality is said to be tight or we simply say that a logarithmic Sobolev
inequality is verified (for short (LSI)). Dimension free gaussian concentration,
hypercontractivity and exponential decay of entropy are directly deduced from such an inequality
explaining the huge interest in it. Note that if a Poincar\'e inequality is valid, a defective
DLSI, via Rothaus's lemma, can be transformed into a (tight) LSI. For all this we refer to
\cite{logsob} or \cite{Wbook}.

Recently, Wang \cite{w00} introduced so called Super-Poincar\'e inequality (say SPI) to study  the
essential spectrum: there exists a non-increasing $\beta\in C(0,\infty)$, all nice $f$ and all
$r>0$
\begin{equation}\label{SP}
\mu(f^2)\le r\,\EE(f,f)+\beta(r)\,\mu(|f|)^2.
\end{equation}
Wang moreover establishes a correspondence between this SPI and defective $F$-Sobolev inequality
(F-Sob) for a proper choice of increasing $F\in[0,\infty[$ with $\lim_\infty F=\infty$, i.e. for
all nice $f$ with  $\mu(f^2)=1$
\begin{equation}\label{FSob}
\mu(f^2F(f^2))\le c_1 \,\EE(f,f)+ c_2.
\end{equation}

More precisely, if \eqref{FSob}holds for some increasing function $F$ satisfying $\lim_{u \to
+\infty} F(u) = +\infty$ and  $\sup_{0<u\leq 1} |u F(u)|<+\infty$, then (SPI) holds with
$\beta(u)=C_1 \, F^{-1}\left(C_2 \, (1+\frac 1u)\right)$ for some well chosen $C_1$ and $C_2$.
Conversely if a (SPI) \eqref{SP} holds, defining $$\xi(t)=\sup_{u>0} \, \left(\frac 1u \, - \,
\frac{\beta(u)}{ut}\right) \, ,$$ a (F-Sob) inequality holds with $$F(u) = \frac {C_1}{u} \,
\int_0^u \, \xi(t/2) dt \, - \, C_2$$ for some well chosen $C_1$ and $C_2$. For details see
\cite{Wbook} Theorem 3.3.1 and Theorem 3.3.3. Note that these results are still available when
$\mu$ is a non-negative possibly non-bounded measure.

In particular an inequality (DLSI) is equivalent to a (SPI) inequality with $\beta(u)=ce^{c'/u}$.
\smallskip

These inequalities and their consequences (concentration of measure, isoperimetry, rate of
convergence to equilibrium) have been studied for diffusions and jump processes by various authors
\cite{w00,BCR1,BCR3,r-w2,CatGui3} under  various conditions.
\medskip

In this paper we shall use Lyapunov type conditions. These conditions are well known to furnish
some results on the long time behavior of the laws of Markov processes (see e.g.
\cite{DMT,Forro,DFG}). The relationship between Lyapunov conditions and functional inequalities of
Poincar\'e type (ordinary or weak Poincar\'e introduced in \cite{r-w}) is studied in details in
the recent work \cite{BCG}. The present paper is thus a complement of \cite{BCG} for the study of
stronger inequalities than Poincar\'e inequality.

We will therefore suppose that $(\XX,d)$ is a Polish space (actually a Riemannian manifold).
Namely we will assume
\begin{itemize}
\item[{\bf(L)}] there is a function $W\ge 1$, a positive function $\phi>\phi_0>0$, $b>0$ and $r_0$
such that
\begin{equation}\label{eqlyap}
{\LL W\over W}\le -\phi+b\,1_{B(o,r_0)}
\end{equation}
where $B(o,r_0)$ is a ball, w.r.t. $d$, with center $o$ and radius $r_0$.
\end{itemize}

The main idea of the paper is the following one: in order to get some super Poincar\'e inequality
for $\mu$ it is enough that $\mu$ satisfies some (SPI) locally and that there exists some Lyapunov
function. In other words the Lyapunov function is useful to extend (SPI) on (say) balls to the
whole space.  General statements are given in section 2.

In particular on nice manifolds the riemanian measure satisfies locally some (SPI), so that an
absolutely continuous probability measure will also satisfy a local (SPI) in most cases. The
existence of a Lyapunov function allows us to get some (SPI) on the whole manifold.

The aim of sections 3 and 4 is to show how one can build such Lyapunov functions, either as a
function of the log-density or as a function of the riemanian distance. In the first case we
improve upon previous results in \cite{KS85,cat5,BCR1,BCR3} among others. In the second case we
(partly) recover and extend some celebrated results: Bakry-Emery criterion for the log-Sobolev
inequality, Wang's result on the converse Herbst argument. In particular we thus obtain similar
results as Wang's one, but for measures satisfying sub-gaussian concentration phenomenon. This
kind of new result can be compared to the recent \cite{BK}.

The main interest of this approach (despite the new results we obtain) is that it provides us with
a drastically simple method of proof for many results. The price to pay is that the explicit
constants we obtain are far to be optimal.
\bigskip


\section{\bf A general result.}\label{secgene}

\subsection{Diffusion case.}\label{subsecdiff}

To simplify we will deal here with the diffusion case: we assume that $\XX=M$ is a $d$-dimensional
connected complete Riemannian manifold, possibly with boundary $\partial M$. We denote by $dx$ the
Riemannian volume element and $\rho(x)=\rho(x,o)$ the Riemannian distance function from a fixed
point $o$. Let $\LL=\Delta-\nabla V.\nabla$ for some $V\in W^{1,2}_{loc}$ such that $Z=\int e^{-V}
d\lambda<\infty$, and  $L$ is self adjoint in $L^2(\mu)$ where $d\mu=Z^{-1}e^{-V}dx$. Note that in
this case, we are in the symmetric diffusion case and the Dirichlet form is given by
$$\EE(f,f)=\int|\nabla f|^2d\mu.$$

We shall obtain (SPI) by perturbing a known super Poincar\'e inequality.

\begin{theorem}\label{thmmain}
Suppose that the Lyapunov condition $(L)$ is verified for some function $\phi$ such that
$\phi(x)\to\infty$ as $\rho(x,o)\to\infty$. Assume also that there exists $T$ locally Lipschitz
continuous on $M$ such that $d\lambda =\exp(-T) \, dx$ satisfies a (SPI) \eqref{SP} with function
$\beta$.
\smallskip

Then \textrm{(SPI)} holds for $\mu$ and some $\alpha : (0,\infty)\to (0,\infty)$ in place of $\beta$.
\smallskip

More precisely, for a family of compact sets $\{A_r\supset B(o, r_0) \}_{r\ge 0}$ such that
$A_r\uparrow M$ as $r\uparrow\infty$, define for $r>0$ : $$ \Phi(r):= \inf_{A_r^c} \phi,\ \ \
\Phi^{-1}(r):= \inf\{s\ge 0: \Phi(s)\ge r\},\ \ \ $$ $$ g(r) := \sup_{\rho(\cdot, A_r)\le 2} |V-T|,\\\
G(r):= \sup_{\rho(\cdot, A_r)\le 2} |\nabla (V-T)|^2,\ \ \ H (r)= \rm{Osc}_{\rho(\cdot, A_r)\le
2}(V-T) .$$

Then we may choose for $s>0$, either $$(1) \quad \alpha(s):= \inf_{\varepsilon \in (0,1)}
\left\{\frac 5 {2\varepsilon} \, \beta\left(\frac{\varepsilon s}{10} \wedge \frac \varepsilon {16}
\wedge \frac{2(1-\varepsilon)}{G\circ \Phi^{-1}(\frac{4b}{\varepsilon}\vee \frac 4
{s\varepsilon})}\right) \, \exp \left(g\circ \Phi^{-1}(\frac{4b}{\varepsilon}\vee \frac 4
{s\varepsilon})\right)\right\},$$ or $$(2) \quad \alpha(s):= 2\exp\Big[2 H\Big( r_0\vee
\Phi^{-1}\big(\frac 4 s \vee \frac{bs} 2\big)\Big)\Big]\beta\big(\frac s 8 e^{- H\circ
\Phi^{-1}(\frac 4 s \vee \frac{bs} 2)}\big).$$
\end{theorem}

\begin{proof}
For $r>r_0$ it holds
\begin{eqnarray*}
\int f^2 d\mu &=&\int_{A_r^c} f^2 d\mu+\int_{A_r} f^2 d\mu\\
&=&\int_{A_r^c} {f^2 \phi\over\phi}d\mu+\int_{A_r} f^2 d\mu\\
&\le& \frac 1{\Phi(r)} \, \int f^2\phi d\mu+\int_{A_r} f^2 d\mu\\
&\le&  \frac 1{\Phi(r)} \, \int f^2 \, \left({-\LL W\over W}\right) d\mu+\left({b\over
\Phi(r)}+1\right) \int_{A_r} f^2 d\mu
\end{eqnarray*}
using our assumption $(L)$.

The proof turns then to the estimation of the two terms in the right hand side of the latter
inequality,  a global term and a local one. For the first term remark, by our assumption on $\LL$
that
\begin{eqnarray*}
\int f^2 \, \left({-\LL W\over W}\right) \, d\mu&=&\int \nabla \left({f^2\over W}\right).\nabla Wd\mu\\
&=&2\int{f\over W}\nabla f.\nabla Wd\mu-\int {f^2|\nabla W|^2\over W^2}d\mu\\&\le& \int|\nabla
f|^2d\mu \, - \, \int \left|\nabla f \, - \, \frac fW \, \nabla W\right|^2 d\mu
\end{eqnarray*}
which leads  to
\begin{equation}\label{cruc-control}
\int f^2 \, \left({-\LL W\over W}\right) \, d\mu\le\int|\nabla f|^2d\mu.
\end{equation}
\smallskip

For the local term we will localize the (SPI) for the measure $\lambda$. To this end, let $\psi$
be a Lipschitz function defined on $M$ such that $\BBone_{A_r} \leq \psi(u) \leq \BBone_{\rho
(.,A_r) \leq 2}$ and $|\nabla \psi| \leq 1$. Writing (SPI) for the function $f \, \psi$  we get
that for all $s>0$
\begin{eqnarray}\label{eqnashlocal}
\int_{A_r} f^2 d\lambda & \leq & \int f^2 \, \psi^2 \, d\lambda \\ & \leq & 2 s \, \int |\nabla
f|^2 \, \BBone_{\rho (.,A_r) \leq 2} d\lambda + 2s \, \int f^2 \BBone_{\rho (.,A_r)
\leq 2}  d\lambda \nonumber \\
&& + \beta(s) \left(\int |f| \, \BBone_{\rho (.,A_r) \leq 2}  d\lambda\right)^2 \, .\nonumber
\end{eqnarray}

To deduce a similar local inequality for $\mu$ we have two methods. For the first one we apply
this inequality to $fe^{-V/2+T/2}$. It yields
\begin{eqnarray*}
\int_{A_r}  f^2 \,  d\mu & = & \int_{A_r} f^2e^{-V+T} d\lambda \\& \leq & 2s \, \int |\nabla f|^2
\BBone_{\rho (.,A_r) \leq 2} \, d\mu +{s\over 2} \, \int f^2 |\nabla(V-T)|^2 \BBone_{\rho (.,A_r)
\leq 2}  \, d\mu\\ && + 2s \, \int f^2 \BBone_{\rho (.,A_r) \leq 2}  d\mu \, + \beta(s) \,
\left(\int |f| \, e^{(V-T)/2} \, \BBone_{\rho (.,A_r)
\leq 2}  d\mu\right)^2\\
\end{eqnarray*}
so that if we choose $s$ small enough so that $s G(r)\le 2(1-\varepsilon)$, we get
\begin{eqnarray}\label{local-control}
\int_{A_r}  f^2 \,  d\mu &\leq& \int 2s \, |\nabla f|^2 d\mu + (1-\varepsilon) \, \int f^2
 d\mu + 2s \, \int f^2 \, d\mu \, \\ && + \,
  \beta(s) \, \exp \left(g(r)\right) \left(\int |f| d\mu\right)^2 \, .
  \nonumber
\end{eqnarray}
Now combine (\ref{cruc-control}) and (\ref{local-control}). On the left hand side we get $$\left(1
\, - \, \left(\frac b {\Phi(r)}+1\right)\left(1 - \varepsilon + 2s\right)\right) \, \int f^2 d\mu
\, .$$ For the coefficient to be larger than $\varepsilon/2$ it is enough that $s \leq
\varepsilon/16$ and $ \Phi(r) \geq 4b/\varepsilon$. Assuming this in addition to $s G(r)\le
2(1-\varepsilon)$ we obtain that for such $s>0$ and $r$,
$$\mu(f^2)\le \frac 2\varepsilon \left({1\over \Phi(r)}+ 5s/2\right)
\mu(|\nabla f|^2)+ \frac {5}{2\varepsilon} \, \beta(s) \, \exp\left(g(r)\right)\mu(|f|)^2 \, .$$
If $t$ is given, it remains to choose first $$r= \Phi^{-1} \Big(\frac{4b}\varepsilon\vee \frac 4
{\varepsilon t}\Big),$$ and then $$ s= \frac {\varepsilon t}{10}\wedge \frac \varepsilon {16}
\wedge \frac{2(1-\varepsilon)}{G(r)},$$to get the first $\alpha(t)$.
\smallskip

The second method is more naive but do not introduce any condition on the gradient of $V$.
\smallskip

Start with
\begin{eqnarray*}
\int  f^2 \, \BBone_{A_r} d\mu & = & \int f^2 \, e^{-V+T} \BBone_{A_r}d\lambda  \leq  e^{-
\inf_{A_r} (V-T)} \, \int f^2 \, \BBone_{A_r} d\lambda \\& \leq & e^{- \inf_{A_r} (V-T)} \Big( 2 s
\, \int |\nabla f|^2 \, \BBone_{\rho(.,A_r)\leq 2} d\lambda + \\ && + 2s \, \int f^2
\BBone_{\rho(.,A_r)\leq 2}
 d\lambda + \beta(s) \left(\int |f| \, \BBone_{\rho(.,A_r)\leq 2}
d\lambda\right)^2\Big) \\ & \leq & e^{- \inf_{A_r} (V-T)} \, e^{\sup_{\rho(.,A_r)\leq 2} (V-T)}
\Big( 2 s \, \int |\nabla f|^2 \, \BBone_{\rho(.,A_r)\leq 2} d\mu + \\ && + 2s \, \int f^2
\BBone_{\rho(.,A_r)\leq 2} d\mu + \beta(s) e^{\sup_{\rho(.,A_r)\leq 2} (V-T)} \left(\int |f| \,
\BBone_{\rho(.,A_r)\leq 2} d\mu\right)^2\Big)
\\ & \leq & e^{Osc_{\rho(.,A_r)\leq 2} (V-T)} \, \left(2 s \, \int |\nabla f|^2  d\mu
 + 2s \, \int f^2
 d\mu\right) + \\ && + e^{2 Osc_{\rho(.,A_r)\leq 2} (V-T)} \, \beta(s) \, \left(\int
|f|  d\mu\right)^2 .
\end{eqnarray*}
If we combine the latter inequality with \eqref{cruc-control} and denote $s'=2s \, e^{
Osc_{\rho(.,A_r)\leq 2} (V-T)}$ we obtain $$\left(1 - \frac{b s'}{\Phi(r)} \right) \, \int f^2 \,
d\mu \, \leq \,$$ $$ \leq \, \left(s'+\frac{1}{\Phi(r)}\right) \, \int |\nabla f|^2 d\mu + e^{2
Osc_{\rho(.,A_r)\leq 2} (V-T)} \, \beta(s' e^{ - Osc_{\rho(.,A_r)\leq 2} (V-T)}/2) \, \left(\int
|f| d\mu\right)^2 \, .$$ Hence, if we choose , $r = \Phi^{-1}(\frac 4s \vee \frac {bs}{2})$ and
$s'=s/4$ we obtain the second possible $\alpha(s)$.
\end{proof}
\medskip

\begin{remark}\label{remcadre}
(1) The previous proof extends immediately to the general case of a ``diffusion'' process with a
``carr\'e du champ'' which is a derivation, i.e. if $\EE(f,f)=\int \Gamma(f,f) d\mu$ for a
symmetric $\Gamma$ such that $\Gamma(fg,h)=f \Gamma(g,h) + g \Gamma(f,h)$ (see \cite{BCG} for more
details on this framework).

(2) For a general diffusion process, say with a non constant diffusion term,  as noted in the
previous remark we have to modify the energy term so that it is no further difficulty and there
are numerous examples where condition {\bf (L)} is verified, i.e. consider $L=a(x)\Delta-x.\nabla$
where $a$ is uniformly elliptic and bounded (consider $W=e^{a|x|^2}$ so $\phi(x)=c|x|^2$). But our
method as expressed here relies crucially on the explicit knowledge of $V$. Note however, that for
the second approach, only an upper bound on the behavior of $V$ over, say, balls is needed, which
can be made explicit in some cases.\hfill $\diamondsuit$

%
\end{remark}
\smallskip

\begin{remark}\label{remlevel}
We may for instance choose $A_r=\bar V_r:=\{x \, ; \, |V-T|(x) < r\}$ (i.e. a level set of
$|V-T|$) provided $|V-T|(x) \to + \infty$ as $\rho(o,x) \to + \infty$. However we have to look at
an enlargement $\bar V^{r+2}=\{x \, ; , \rho(x,\bar V_r) < 2\}$ (not the level set of level
$r+2$).

If we want to replace $\bar V^{r+2}$ by the level set $\bar V_{r+2}$ we have to modify the proof,
choosing some ad-hoc function $\psi$ which is no more 1-Lipschitz. It is not difficult to see that
we have to modify \eqref{eqnashlocal} and what follows, replacing $1$ (the $1$ of $1$-Lipschitz)
by $\sup_{\bar V_{r+2}} |\nabla (V-T)|^2$. So we have to modify the condition on $s$ in (1) of the
previous theorem, i.e.
\begin{equation}\label{eqmodif}
\frac{2}{\inf_{(\bar V_r)^c} \phi} \leq \varepsilon \, s \, \leq \, \frac{2}{\inf_{(\bar V_r)^c}
\phi} + \frac{2 \, (1-\varepsilon)}{\sup_{\bar V_{r+2}} |\nabla(V-T)|^2} \, ,
\end{equation}
i.e. we get the same result as (1) but with $\Phi(r)=\inf_{(\bar V_r)^c} \phi$, $g(r)= r+2$ and
$G(r)=\sup_{\bar V_{r+2}} |\nabla(V-T)|^2$.

The second case (2) cannot (easily) be extended in this direction. \hfill $\diamondsuit$
\end{remark}
\medskip

Actually one can derive a lot of results following the lines of the proof, provided some ``local''
(SPI) is satisfied. Here is the more general result in this direction.

\begin{theorem}\label{thmmainbis}
In theorem \ref{thmmain} define $\lambda_{A_r}(f)=\lambda(f \BBone_{A_r})$ where $A_r$ is an
increasing family of open sets such that $\bigcup_r A_r= M$. Given two such families $A_r
\subseteq B_r$, assume that for all $r$ large enough the following local (SPI) holds,
\begin{equation}\label{eqlocalspi}
\lambda_{A_r}(f^2) \leq s \lambda_{B_r}(|\nabla f|^2) + \beta_r(s)
\left(\lambda_{B_r}(|f|)\right)^2 \, .
\end{equation}
Then the conclusions of theorem \ref{thmmain} are still true if we replace $\rho(.,A_r)\leq 2$ by
$B_r$ and $\beta(s)$ by $\beta_{r(s)}(s)$ with $r(s)= \Phi^{-1} \Big(\frac{4b}\varepsilon\vee
\frac 4 {\varepsilon s}\Big)$ for each given $\varepsilon$ in case (1) and $r(s) = \Phi^{-1}(\frac
4s \vee \frac {bs}{2})$ in case (2).
\end{theorem}

\bigskip

\subsection{General case}\label{subsecjump}

We consider here the case of general Markov processes on a Manifold $M$, with a particular care to
jump processes. Indeed, a crucial step in the previous proof is to prove (\ref{cruc-control}) and it has been
made directly taking profit of the gradient structure, but it can be proved in greater
generality. However the second part relying on a perturbation approach seems more difficult. We
therefore introduce a local Super-Poincar\'e inequality.

\begin{theorem}
Suppose that the Lyapunov condition $(L)$ is verified for some function $\phi$ such that
$\phi(x)\to\infty$ as $\rho(x,o)\to\infty$. Assume also the following family of local Super
Poincar\'e inequality holds for $\mu$:
 for a family of compact sets $\{A_r\supset B(o, r_0) \}_{r\ge 0}$ such that
$A_r\uparrow M$ as $r\uparrow\infty$, there exists $\beta(r,\cdot)$ such that for all nice $f$ and $s>0$
\begin{equation}
\mu(f^21_{A_r})\le s\mathcal{E}(f,f)+\beta(r,s)\mu(|f|)^2.
\end{equation}
Then, denoting
$$ \Phi(r):= \inf_{A_r^c} \phi,\ \ \
\Phi^{-1}(r):= \inf\{s\ge 0: \Phi(s)\ge r\},\ \ \ $$
$\mu$ verifies a Super Poincar\'e inequality with function
$$\alpha(s)=\beta(\Phi^{-1}(2/s),s/2).$$
\end{theorem}
The proof relies on a simple optimization procedure between  the weighted energy term and the
local Super Poincar\'e inequality. We then only have to prove (\ref{cruc-control}) which is done
by the following large deviations argument.

\begin{lemma}\label{lem-controlgen} For every continuous  function $U\ge 1$ such
that $-\LL U/U$ is bounded from below,

\begin{equation}\label{cont1}
\int -\frac{\LL U}{U} f^2 d\mu \le \EE(f,f), \ \forall
f\in\D(\EE).
\end{equation}
\end{lemma}

\begin{proof}  Remark that
$$
N_t = U(X_t)\exp\left(-\int_0^t \frac{\LL U}{U}(X_s)ds\right)
$$
is a $\P_\mu$-local martingale. Indeed, let
$A_t:=\exp\left(-\int_0^t \frac{\LL U}{U}(X_s)ds\right)$, we have
by Ito's formula,

$$
dN_t =A_t [dM_t(U)+ \LL U(X_t) dt] -\frac{\LL U}{U}(X_t) A_t
U(X_t)dt=A_t dM_t(U).
$$
Now let $\beta:=(1+U)^{-1}d\mu/Z$ ($Z$ being the normalization constant). $(N_t)$ is also local
martingale, then a super-martingale w.r.t. $\P_\beta$. We so get
$$
\E^{\beta} \exp\left(-\int_0^t \frac{\LL U}{U}(X_s)ds\right)\le
\E^{\beta} N_t \le \beta(U)<+\infty.
$$
Let $u_n:=\min\{-\LL U/U, n\}$. The estimation above implies
$$
F(u_n):=\limsup_{t\to\infty} \frac 1t\log
\E^{\nu}\exp\left(-\int_0^t u_n(X_s)ds\right)\le0.
$$
On the other hand by the lower bound of large deviation in \cite[Theorem B.1, Corollary B.11]{wu1}
and Varadhan's Laplace principle, defining $I(\nu|\mu)=\EE(\sqrt{d\nu/d\mu},\sqrt{d\nu/d\mu})$
$$
F(u_n)\ge \sup\{\nu(u_n)-I(\nu|\mu);\ \nu\in M_1(E)\}.
$$
Thus $\int u_n d\nu\le I(\nu|\mu)$, which yields to (by letting
$n\to\infty$ and monotone convergence)

\begin{equation}\label{cont2}
\int -\frac {\LL U} U d\nu \le I(\nu|\mu),\ \forall \nu\in M_1(E).
\end{equation}
That is equivalent to (\ref{cont1}) by the fact that $\EE(|f|,
|f|)\le \EE(f,f)$ for all $f\in\D(\EE)$.
 \end{proof}
\medskip

We will discuss examples on jump processes in future research, see however \cite[Th. 3.4.2]{Wbook} for results in this direction.


\section{\bf Examples in $\R^n$.}

We use the setting of the subsection \ref{subsecdiff} (or of the remark \ref{remcadre}) but in the
euclidean case $M=\R^n$ for simplicity. Hence in this section $\lambda$ is the Lebesgue measure,
i.e we have $T=0$. Recall that $d\mu = Z^{-1} \, e^{-V} \, dx$. It is well known that $\lambda$
satisfies a  (SPI) with $\beta(s)=c_1 + c_2 s^{-n/2}$. However it is interesting to have some
hints on the constants (in particular dimension dependence). It is also interesting (in view of
Theorem \ref{thmmainbis}) to prove (SPI) for subsets of $\mathbb R^n$.

Hence we shall first discuss the (SPI) for $\lambda$ and its restriction to subsets. Since we want
to show that the Lyapunov method is also quite quick and simple in many cases, we shall also
recall the quickest way to recover these (SPI) results.
\smallskip

\subsection{Nash inequalities for the Lebesgue measure.} \label{subsecnash}

Let $A$ be an open connected domain with a smooth boundary. For simplicity we assume that
$A=\{\psi(x)\leq 0\}$ for some $C^2$ function $\psi$ such that $|\nabla \psi|^2(x) \geq a > 0$ for
$x \in \partial A = \{\psi=0\}$. It is then known that one can build a Brownian motion reflected
at $\partial A$, corresponding to the heat semi-group with Neumann condition. Let $P_t^N$ denote
this semi-group, and denote by $p_t^N$ its kernel. Recall the following

\begin{proposition}\label{propultra}
The following statements are equivalent
\begin{enumerate}
\item[(\ref{propultra}.1)] for all $0<t \leq 1$ and all $f\in \L^2(A,dx)$, $$\parallel P_t^N
f\parallel_\infty \, \leq \, C_1 \, t^{-n/4} \,
\parallel f\parallel_{\L^2(A,dx)} \, ,$$
\item[(\ref{propultra}.2)] (provided $n>2$) for all $f\in C^{\infty}(\bar A)$, $$\parallel
f\parallel^2_{\L^{2n/n-2}(A,dx)} \, \leq \, C_2 \, \left(\int_A \, |\nabla f|^2 dx + \int_A \, f^2
dx\right) \, , $$ \item[(\ref{propultra}.3)] for all $f\in C^{\infty}(\bar A)$, $$\parallel
f\parallel^{2+4/n}_{\L^2(A,dx)} \, \leq \, C_3 \, \left(\int_A \, |\nabla f|^2 dx + \int_A \, f^2
dx\right) \, \parallel f\parallel^{4/n}_{\L^1(A,dx)} \, ,$$ \item[(\ref{propultra}.4)] the (SPI)
inequality
$$\int_A f^2 dx \, \leq \, s \, \int_A \, |\nabla f|^2 dx  + \beta(s) \, \left(\int_A |f|
dx\right)^2$$ holds with $\beta(s)= C_4 (s^{-n/2} + 1)$.
\end{enumerate}
Furthermore any constant $C_i$ can be expressed in terms of any other $C_j$ and the dimension $n$.
\end{proposition}

These results are well known. they are due to Nash, Carlen-Kusuoka-Stroock (\cite{CKS}) and
Davies, and can be found in \cite{Dav} section 2.4 or \cite{SCsob}. Generalizations to other
situations (including general forms of rate functions $\beta$) can be found in \cite{Wbook}
section 3.3.

If $A=\R^n$  (\ref{propultra}.1) holds (for all $t$) with $C=(2\pi)^{-n/2}$ and $\alpha=n/2$,
yielding a (SPI) inequality with
\begin{equation}\label{eqspilebesgue}
\beta(s) = c_n \, s^{- n/2} \, = \, \left(\frac{1}{4 \pi}\right)^{n/2} \,  s^{- n/2} \, ,
\end{equation}
which is equivalent, after optimizing in $s$, to the Nash inequality
\begin{equation}\label{eqnashlebesgue}
\parallel f \parallel_2^{2 + 4/n} \, \leq \, C_n \, \left(\int |\nabla f|^2 dx\right) \, \parallel f
\parallel_1^{4/n} \, ,
\end{equation}
with $C_n= 2 (1+2/n) \, (1+ n/2)^{2/n} \, (1/8\pi)^{n/4}$ .
\smallskip

For nice open bounded domains in $\R^n$, as we consider here, (\ref{propultra}.2) is a well known
consequence of the Sobolev inequality in $\R^n$ (see e.g. \cite{Dav} Lemma 1.7.11 and note that
the particular cases $n=1,2$ can be treated by extending the dimension (see \cite{Dav} theorem
2.4.4)). But we want here to get some information on the constants. In particular, when $A$ is the
level set $\bar V_r$ we would like to know how $\beta_r$ depends on $r$.
\medskip

\begin{remark}
If $n=1$, we have an explicit expression for $p_t^N$ when $A=(0,r)$, namely $$p^N_t(x,y) = (2
\pi t)^{-n/2} \, \sum_{k\geq 0} \, \left(\exp \left( - \, \frac{(x-y - 2k r)^2}{2t}\right) + \exp
\left( - \, \frac{(x+y + 2k r)^2}{2t}\right)\right) \, .$$ It immediately follows that
\begin{equation}\label{eqneumann1}
\sup_{x,y \in (0,r)} \, p^N_t(x,y) \leq (2 \pi t)^{-n/2} \, \left(2 + \, \sum_{k\geq 1} \,
\left(\exp \left( - \, \frac{((2k-1) r)^2}{2t}\right) + \exp \left( - \, \frac{(2k
r)^2}{2t}\right)\right)\right) \, ,
\end{equation}
so that, using translation invariance, for any interval $A$ of length $r>r_0$ and for $0<t \leq
1$, $\sup_{x,y \in A} \, p^N_t(x,y) \leq c(r_0) \, (2 \pi t)^{-n/2}$.  Hence (\ref{propultra}.1)
is satisfied, and a (SPI) inequality holds in $A$ with the same function $\beta_r(s)=c_B \,
(s^{-n/2}+1)$ \underline{independently} on $r>r_0$. By tensorization, the result extends to any
cube or parallelepiped in $\R^n$ with edges of length larger than $r_0$. \hfill $\diamondsuit$
\end{remark}
\medskip

If we replace cubes by other domains, the situation is more intricate. However in some cases one
can use some homogeneity property. For instance, for $n>2$ we know that (\ref{propultra}.2) holds
for the unit ball with a constant $C_2$ (for $n=2$ we may add a dimension and consider a cylinder
$B_2(0,1)\otimes \R$ as in \cite{Dav} theorem 2.4.4). But a change of variables yields $$\parallel
f\parallel^2_{\L^{2n/n-2}(B(0,r),dx)} \, \leq \, C_2 \, \left(\int_{B(0,r)} \, |\nabla f|^2 dx +
r^{-2} \int_{B(0,r)} \, f^2 dx\right) \, , $$ so that for $r\geq 1$ (\ref{propultra}.2) holds in
the ball of radius $r$ with a constant $C_2$ independent of $r$.
\smallskip

The previous argument extends to $A=\bar V_r$ provided for $r \geq r_0$, $\bar V_r$ is
star-shaped, in particular it holds if $V$ is convex at infinity. This is a direct consequence of
the coarea formula (see e.g \cite{EG} proposition 3 p.118). Indeed if $f$ has his support in an
annulus $r_0 < r_1 < V(x) < r_2$ the surface measure on the level sets $\bar V_r$ is an image of
the surface measure on the unit sphere. This is immediate since the application $x \mapsto (V(x),
\frac{x}{|x|})$ is a diffeomorphism in this annulus. Hence for such $f$'s the previous homogeneity
property can be used. For a given $r>r_0$ large enough, it remains to cover $\bar V_r$ by such an
annulus and a large ball (such that the ball contains $\bar V_{r_0}$ and is included in $\bar
V_r$) and to use a partition of unity related to this recovering. We thus get as before that for
$r$ large enough, $C_2$ can be chosen independent of $r$.
\smallskip

For general domains $A$, recall that (\ref{propultra}.2) holds true if $A$ satisfies the
``extension property'' of the boundary, i.e.  the existence of a continuous extension operator $E:
\W^{1,2}(A) \rightarrow \W^{1,2}(\R^n)$. If this extension property is true, (\ref{propultra}.2)
is true in $A$ with a constant $C_2$ depending only on $n$ and the operator norm of $E$ (see
\cite{Dav} proposition 1.7.11).

If $A=\bar V_r$ is bounded, as soon as $\nabla V$ does not vanish on $\partial A$, the implicit
function theorem tells us that for all $x\in \partial A$ one can find an open neighborhood $v_x$
of $x$, an index $i_x$ and a 2-Lipschitz function $\phi_x$ defined on $v_x$ such that $$v_x \cap
A=v_x \cap \{\phi(y_1,...,y_{i_x-1},y_{i_x+1}, ...,y_n)<y_{i_x}\} \, .$$ To this end choose $i_x$
such that $|\partial_{i_x}V| (x) \geq |\partial_{j} V|(x)$ for all $j=1,...,n$, so that, for $y
\in \partial A$ neighboring $x$,  $2 |\partial_{i_x}V| (y) \geq |\partial_{j} V|(y)$, and the
partial derivative of the implicit function $\phi$ given by the ratio $\partial_{j} V(y)/
\partial_{i_x}V (y)$ is less than 2 in absolute value.

By compactness we may choose a finite number $Q$ of points such that $\bigcup_{j=1,...,Q} v_{x_j}
\supset
\partial A$. Hence we are in the situation of \cite{Dav} proposition 1.7.9. This property implies the
extension property but with some extension operator $E$ whose norm depends on two quantities :
first the maximal $\varepsilon>0$ such that for all $x \in \partial A$, $B(x,\varepsilon)
\subseteq v_{x_j}$ for some $j=1,...,Q$; second, the maximal integer $N$ such that any $x\in
\partial A$ belongs to at most $N$ such $v_{x_j}$'s. This is shown in \cite{Stein} p.180-192.

Actually an accurate study of Stein's proof (p. 190 and 191) shows that $\parallel E\parallel \leq
C(n) \, (N/\varepsilon)$ (recall that we have chosen $\phi$ 2-Lipschitz).
\smallskip

Now assume that
\begin{equation}\label{eqbord}
\textrm{ there exist $R >0$, $v>0$ , $k \in \N$ such that for }  |x|\geq R \, , \, |\nabla V(x)|
\geq v > 0 \, .
\end{equation}

Then it is easy to check that for $A=\bar V_r$ it holds $\varepsilon \leq \varepsilon_0 = c(v,R,n)
\, \theta^{-1}(r)$ with
\begin{equation}\label{eqsecondordre}
\theta(r)=\sup_{x \in \partial \bar V_r} \max_{i,j=1,...,n} \left|\frac{\partial^2 V}{\partial x_i
\,
\partial x_j}(x)\right| \,  .
\end{equation}

But $\varepsilon_0$ being given, it is well known that one can find a covering of $A$ by balls of
radius $\varepsilon_0/2$ such that each $x \in\bar V_r$ belongs to at most $N= c^n$ such balls for
some universal $c$ large enough. Hence $N$ can be chosen as a constant depending on the dimension
only. It follows that

\begin{proposition}\label{propbord}
If \eqref{eqbord} is satisfied, the (SPI) (\ref{propultra}.4) holds with $A=\bar V_r$, $\theta$
defined by \eqref{eqsecondordre} and
$$\beta_r(s)= C(n) \, \theta^n(r) \, (1+s^{- n/2}) \, .$$
\end{proposition}

For the computation of $\beta_r$ we used \cite{Dav} lemma 1.7.11 which says that $C_2=c(n)
\parallel E\parallel^2$ and \cite{Dav} proof of theorem 2.4.2 p.77 which yields a logarithmic
Sobolev inequality with $\beta(\varepsilon)= -(n/4) \log \varepsilon + (n/4) \log (C_2 n/4)$
together with \cite{Dav} corollary 2.2.8 which gives $C_1= c(n) \, C_2^{n/4}$. Finally the proof
of \cite{Dav} theorem 2.4.6 gives $\beta(s)= C_1^2 \, (1+(s/2)^{- n/2})$.
\medskip

Proposition \ref{propbord} gives of course the worse result and in many cases one can expect a
much better behavior of $\beta_r$ as a function of $r$. In particular in the homogeneous case we
know the result with a constant independent of $r$.

\begin{remark}
Another possibility to get (SPI) in some domain $A$, is to directly prove the Nash inequality
(\ref{propultra}.3). One possible way to get such a Nash inequality is to prove some
Poincar\'e-Sobolev inequality. The case of euclidean balls is well known.

According to \cite{SCsob} theorem 1.5.2, for $n>2$, with $p=2$ and $s=2n/(n-2)=2^*$ therein, for
all $r>0$ and all ball $B_r$ with radius $r$, if $\lambda_r$ is the Lebesgue measure on $B_r$ and
$\bar f_r = (1/Vol(B_r)) \int_{B_r} \, f dx$ we have
\begin{equation}\label{eqsobpoinc}
\lambda_r\left(|f - \bar f_r|^{\frac{2n}{n-2}}\right)^{\frac{n-2}{2n}} \, \leq \, C_n \,
\lambda_r\left(|\nabla f|^2\right)^{\frac 12} \, ,
\end{equation}
so that using first Minkowski, we have
\begin{equation}\label{eqsobpoincprime}
\lambda_r\left(|f|^{\frac{2n}{n-2}}\right)^{\frac{n-2}{2n}} \, \leq \, C_n \,
\lambda_r\left(|\nabla f|^2\right)^{\frac 12} \, + \, \frac{1}{Vol(B_r)} \, \lambda_r(|f|) \, ,
\end{equation}
and finally using H\"{o}lder inequality and Cauchy-Schwarz inequality we get the local Nash
inequality
\begin{eqnarray}\label{eqnashloclebesgue}
\quad \lambda_r(|f|^2) & \leq & \left(\lambda_r(|f|)\right)^{4/(n+2)} \, \left(C_n \,
\lambda_r\left(|\nabla f|^2\right)^{\frac 12} \, + \, \frac{1}{Vol(B_r)} \,
\lambda_r\left(|f|\right)\right)^{2n/(n+2)} \, ,\\ & \leq & \left(\lambda_r(|f|)\right)^{4/(n+2)}
\, \left(C_n \, \lambda_r\left(|\nabla f|^2\right)^{\frac 12} \, + \, \frac{1}{\sqrt{Vol(B_r)}} \,
\lambda_r\left(|f|^2\right)^{\frac 12} \right)^{2n/(n+2)} \, . \nonumber
\end{eqnarray}
Again, for $r>r_0$ we get a Nash inequality hence a (SPI) inequality independent of $r$ with
$\beta_r(s)=c_n (1+s^{-n/2})$.

Notice that \eqref{eqsobpoinc} is scale invariant, i.e. if it holds for some subset $A$, it holds
for the homotetic $r A$ ($r>0$) with \underline{the same constants}. That is why the constants do
not depend on the radius for balls. If we replace a ball by a convex set, the classical method of
proof using Riesz potentials (see e.g. \cite{SCsob} or \cite{Dav} lemma 1.7.3) yields a similar
results but with an additional constant, namely $diam^n(A)/Vol(A)$, so that if $V$ is a convex
function the constant we obtain with this method in \eqref{eqsobpoinc} for $\bar V_r$ may depend
on $r$.
\smallskip

Actually the Sobolev-Poincar\'e inequality \eqref{eqsobpoinc} extends to any John domain with a
constant $C$ depending on the dimension $n$ and on the John constant of the domain. This result is
due to Bojarski \cite{Boj} (also see \cite{Haj} for another proof and \cite{BK96} for a converse
statement). Actually a John domain satisfies some chaining (by cubes or balls) condition which is
the key for the result (see the quoted papers for the definition of a John domain and the chaining
condition). But an explicit calculation of the John constant is not easy. \hfill $\diamondsuit$
\end{remark}
\medskip

\subsection{Typical Lyapunov functions and applications.}\label{subseclyapchoice}

We here specify classes of natural Lyapunov function: function of the potential or of the
distance. As will be seen, it gives new practical conditions for super-Poincar\'e inequality and
for logarithmic Sobolev inequality.

First, since $W \geq 1$ we may write $W=e^U$ so that condition (L) becomes
\begin{equation}\label{eqcond}
\Delta U + |\nabla U|^2 - \nabla U . \nabla V + \phi \, \leq \, 0 \quad \textrm{ ``at infinity''.}
\end{equation}
\medskip

\subsubsection{Lyapunov function $e^{aV}$}\label{seclyapchoice}

Test functions $e^{aV}$ for $a<1$ are quite natural in that they are the limiting case  for the
spectral gap (see \cite{BCG}).  Indeed, $\mu(e^{aV})$ is finite if and only if $a<1$ and a drift
condition such that
$$\LL W\le -\lambda W+b 1_C$$
formally implies by integration by $\mu$, that $\mu(W)$ is finite. So in a sense, $e^{aV}$ are the
``largest'' possible  Lyapunov functions.

Hence, if $W=e^{aV}$, $\frac{\LL W}{W} = a \, \left(\Delta V-(1-a)|\nabla V|^2\right)$. Introduce
the following conditions
\begin{enumerate}\stepcounter{equation}\eqlabel{\theequation}\label{condcat}
\item[(\theequation.1)] $V(x) \to +\infty$ as $|x| \to +\infty$, \item[(\theequation.2)] there
exist $0<a_0<1$, a non-decreasing function $\eta$ with $\eta(u) \to +\infty$ as $u \to +\infty$
and a constant $b_0$ such that $$(1-a_0) |\nabla V|^2 - \Delta V \geq \eta(V) + b_0 \,
\BBone_{|x|<R}\, ,$$ \item[(\theequation.3)] $\limsup_{|x| \to +\infty} \left(\eta(V(x))/|\nabla
V(x)|^2\right) < +\infty$.
\end{enumerate}

Then for $0<a<a_0$ condition (L) is satisfied with $$\phi=a(a_0-a) |\nabla V|^2 + a \eta(V) \, .$$
In addition $\inf_{(\bar V_r)^c} \, \phi(V) \leq c \, \sup_{\bar V_{r+2}} |\nabla V|^2$.

Following remark \ref{remlevel} (we choose arbitrarily  $\varepsilon=1/2$ here) we obtain for some
constant $c$
\begin{equation}\label{eqcat1}
 \int \, f^2 d\mu \leq \, s \, \int |\nabla f|^2 d\mu + c \, \left(1 + \sup_{\bar
V_{2+\eta^{-1}(c/s)}} |\nabla V|^2\right)^{n/2} \, e^{\eta^{-1}(c/s)} \, \left(\int |f|
d\mu\right)^2 \, .
\end{equation}
 We thus clearly see that to get an explicit (SPI) we need to control the
gradient  $\nabla V$ on the level sets of $V$.
\smallskip

If instead of using theorem \ref{thmmain}.(1) we want to use theorem \ref{thmmain}.(2) or more
precisely theorem \ref{thmmainbis} we have to use proposition \ref{propbord}. Hence since
\eqref{eqbord} is satisfied we obtain for $s$ small enough
\begin{equation}\label{eqcat2}
 \int \, f^2 d\mu \leq \, s \, \int |\nabla f|^2 d\mu + C \,
\theta^n(\eta^{-1}(c/s)) \, e^{2 \eta^{-1}(c/s)} \, \left(1 + s^{-n/2} e^{n
\eta^{-1}(c/s)/2}\right) \, \left(\int |f| d\mu\right)^2 \, .
\end{equation}
We have obtained

\begin{theorem}\label{propcat}
Assume that (\ref{condcat}.1), (\ref{condcat}.2), (\ref{condcat}.3) are satisfied. Then $\mu$ will
satisfy a (SPI) inequality with function $\beta$ in one of the following cases
\begin{enumerate}
\item[(\ref{propcat}.1)] \quad for $|x|$ large enough, $|\nabla V|(x) \leq \gamma (V(x))$ and
$\beta(s)= C (1 + e^{\eta^{-1}(c/s)} \, \gamma^n(\eta^{-1}(c/s)))$, \item[]
\item[(\ref{propcat}.2)] \quad for $|x|$ large enough $\left|\frac{\partial^2 V}{\partial x_i \,
\partial x_j}(x)\right| \leq \theta(V(x))$ and $$\beta(s)= C \left(1 + \theta^n(\eta^{-1}(c/s))
 \, s^{-n/2} e^{(n+4) \eta^{-1}(c/s)/2}\right) \, .$$
\end{enumerate}
\end{theorem}
\smallskip

\begin{remark}\label{remcat}
If $\eta(u)=u$ we thus obtain that $\mu$ satisfies a (defective) logarithmic Sobolev inequality
provided either $\gamma(u) \leq e^{Ku}$ or $\theta(u) \leq e^{Ku}$. But (\ref{condcat}.1) and
(\ref{condcat}.2) imply that $\mu$ satisfies a Poincar\'e inequality (see e.g. \cite{BCG}
corollary 4.1). Hence using Rothaus lemma we get that $\mu$ satisfies a (tight) logarithmic
Sobolev inequality.

Conditions (\ref{condcat}.1) and (\ref{condcat}.2), with $\eta(u)=u$, appear in \cite{KS85} where
the authors show that they imply the hypercontractivity of the associated symmetric semi-group,
hence a logarithmic Sobolev inequality by using Gross theorem. In particular the additional
assumptions on the first or the second derivatives do not seem to be useful. Another approach
using Girsanov transformation was proposed in \cite{cat5} for $a_0=1/2$, again without the
technical assumptions on the derivatives. This approach extends to more general processes with a
``carr\'e du champ''.

Here we directly get the logarithmic Sobolev inequality without using Gross theorem, but with some
conditions on $V$.
\smallskip

The advantage of theorem \ref{propcat} is that it furnishes an unified approach of various
inequalities of $F$-Sobolev type. In \cite{BCR1} conditions (\ref{condcat}.1) and
(\ref{condcat}.2) are used (for particular $\eta$'s) to get the Orlicz-hypercontractivity of the
semi-group hence a $F$-Sobolev inequality thanks to the Gross-Orlicz theorem proved therein. The
use of this theorem requires some quite stringent conditions on $\eta$ but covers the case
$\eta(u)=u^{\alpha}$ for $1<\alpha<2$, yielding a $F$-Sobolev inequality for
$F(u)=\log_+^\alpha(u)$ (more general $F$ are also studied in \cite{BCR3} section 7). Note that in
theorem \ref{propcat} we do no more need the restriction $\alpha<2$, but we need some control on
the growth on $\gamma$ or $\theta$, namely we need again $\gamma(u) \leq e^{K u}$ (the same for
$\theta$).

We also obtain a larger class of $F$-Sobolev inequalities thanks to the correspondence between
$F$-Sobolev and (SPI) recalled in the introduction. The reader is referred to \cite{Wbook} section
5.7 for related results in the ultracontractive case. \hfill $\diamondsuit$
\end{remark}

\medskip

\subsubsection{Lyapunov function $e^{a|x|^b}$}\label{secdistchoice}

If we try to use $W=e^{a|x|^b}$ we are led to choose $$\phi(x) = ab |x|^{b-2} \psi(x)$$ with
$$\psi(x) = x.\nabla V - \left(n +(b-2) + ab |x|^{b}\right)$$ provided the latter quantities
are bounded from below by a positive constant for $|x|$ large enough.

Introduce now the standard curvature assumption
\begin{equation}\label{eqcurvature}
\textrm{ for all $x$, } \quad Hess V (x) \geq c_0 \, Id
\end{equation}
for some $\rho \in \R$. This assumption allows to get some control on $x.\nabla V$ namely
\begin{lemma}\label{lemcurvature}
If \eqref{eqcurvature} holds, $x.\nabla V (x) \geq V(x) - V(0) + c_0 \, |x|^2/2$.
\end{lemma}
\begin{proof}
Introduce the function $g(t)=t \, x.\nabla V(tx)$ defined for $t\in [0,1]$. \eqref{eqcurvature}
implies that $g'(t) \geq x.\nabla V(tx) + tc_0 |x|^2$ and the result follows by integrating the
latter inequality between $0$ and $1$.
\end{proof}
\smallskip

We may thus state
\begin{proposition}\label{propcurvature}
Assume that \eqref{eqcurvature} is satisfied. Then one can find positive constants $c,C$ such that
$\mu$ satisfies some (SPI) with function $\beta$ (given below for $s$ small enough) in the
following cases
\begin{enumerate}
\item[(\ref{propcurvature}.1)] \quad $c_0 \geq 0$,  $V(x) \geq c' |x|^b$ for $|x|$ large enough
some $c'>0$ and $b>1$, $$\beta(s)= C \, e^{c (1/s)^{\frac{b}{2\left((b-1)\wedge 1\right)}}} \, .$$
\item[(\ref{propcurvature}.2)] \quad $c_0 \geq 0$,  $d' |x|^{b'} \geq V(x) \geq c' |x|^b$ for
$|x|$ large enough some $d', c'>0$ and $b' \geq b > 1$, $$\beta(s)= C \, e^{c
(1/s)^{\frac{b'}{b'+b -2}}} \, .$$ \item[(\ref{propcurvature}.3)] \quad $c_0 \leq 0$, for $|x|$
large enough, $V(x) \geq (\varepsilon - c_0/2)  |x|^2$ for some $\varepsilon
>0$, and $$\beta(s)= C \, e^{c (1/s)} \, .$$
\item[(\ref{propcurvature}.4)]  \quad $c_0 \leq 0$, for $|x|$ large enough, $d' |x|^{b'} \geq
V(x) \geq c' |x|^b$ and $\beta$ as in (\ref{propcurvature}.2).
\end{enumerate}
\end{proposition}
\begin{proof}
In all the proof $D$ will be an arbitrary positive constant whose value may change from place to
place. All the calculations are assuming that $|x|$ is large enough.

Consider first case 1. Choosing $a$ small enough and using lemma \ref{lemcurvature}, we see that
$\phi(x) \geq D \, |x|^{b-2} \, V(x)$. If $b\geq 2$ we thus have $\phi(x) \geq D \, V(x)$ while
for $b<2$, $\phi(x) \geq D \, V^{\frac{2(b-1)}{b}}(x)$ for large $|x|$ according to the
hypothesis. For $\phi$ to go to infinity at infinity, $b>1$ is required. In particular on the
level sets $\bar V_r$ we have either $\phi(x) \geq D \, r$ or $\phi(x) \geq D \, r^{2(b-1)/b}$.

Now since the level sets $\bar V_r$ are convex, we know that some Nash inequality holds on $\bar
V_r$ according to the discussion in the previous subsection. We may thus use theorem
\ref{thmmainbis} in the situation (2) of theorem \ref{thmmain}. Choosing $s= d / r$ or $s=
d/r^{2(b-1)}{b}$ for some well chosen $d$ yields the result with an extra factor $s^{-k}$ for some
$k>0$. This extra term can be skipped just changing the constants in the exponential term.
\smallskip

Case 2 is similar but improving the lower bound for $\phi$. Indeed since $D |x| \geq  \,
V^{1/b'}(x)$, $\phi(x) \geq D \, V^{\frac{b'+b-2}{b'}}(x)$. It allows us to improve $\beta$.
\smallskip

Let us now consider Case 3. Since $b=2$, our hypothesis implies that for $2a < \varepsilon$, $\phi
\geq D V$. But the curvature assumption implies that the level sets of $x \mapsto H(x)= V(x)+c_0
|x|^2/2$ are convex. Since $V(x) \geq D |x|^2$, one has $c r \leq V(x) \leq r$ if $x \in \bar H_r$.
We may thus mimic case 1, just replacing $\bar V_r$ by $\bar H_r$. Case 4 is similar to the
previous one just improving the bound on $\phi$ as in case 2.
\end{proof}
\smallskip

\begin{corollary}\label{corcurvature}
\begin{enumerate}
\item[(1)] \quad If (\ref{propcurvature}.3) holds, $\mu$ satisfies a logarithmic Sobolev
inequality. \item[(2)] \quad If (\ref{propcurvature}.1) holds with $b=2$, $\mu$ satisfies a
logarithmic Sobolev inequality. In particular if $\rho
> 0$, $\mu$ satisfies a logarithmic Sobolev inequality (Bakry-Emery criterion). \item [(3)] \quad
If (\ref{propcurvature}.1) holds for some $1<b<2$, $\mu$ satisfies a $F$-Sobolev inequality with
$F(u)=\log_+^{2(1-(1/b))}(u)$.
\end{enumerate}
\end{corollary}

The first statement of the theorem is reminiscent to Wang's improvement of the Bakry-Emery
criterion, namely if $\int \, \int \, e^{(-\rho + \varepsilon)|x-y|^2} \, \mu(dx) \, \mu(dy) <
+\infty$, $\mu$ satisfies a logarithmic Sobolev inequality. Our statement is weaker since we are
assuming some uniform behavior. The third statement can thus be seen as an extension of Wang's
result to the case of $F$-Sobolev inequalities interpolating between Poincar\'e inequality and
log-Sobolev inequality. These inequalities are related to the Latala-Oleskiewicz interpolating
inequalities \cite{LO00}, see \cite{BCR1} for a complete description.

It should be interesting to improve (3) in the spirit of Wang's concentration result. See
\cite{Kol,BK} for a tentative involving modified log-Sobolev inequalities introduced in \cite{GGM1} and mass transport.

\bigskip

\section{The general manifold case}

In fact as one guesses, the main point is to get the additional Super Poincar\'e inequality, local
as developed in Section \ref{subsecnash}, or global (and then using the localization technique
already mentioned). It is of course a fundamental field of research which encompasses the scope of
the present paper. We may however use our main results Theorem \ref{thmmain} and Theorem
\ref{thmmainbis}, with the same Lyapunov functionals as developed in Sections \ref{seclyapchoice}
and \ref{secdistchoice}, replacing of course the euclidean distance by the Riemannian distance
(w.r.t a fixed point), at least in two main cases.

According to \cite{Croke}, if the injectivity radius of $M$ is positive then (\ref{SP}) holds for
$T=0$ and $\beta(s)= c_1 +c_2 s^{-d/2}$ for some constants $c_1, c_2>0$; if in particular  the
injectivity is infinite, then one may take $c_1=0$, \cite{w00} page 225.

 Next, if the Ricci
curvature of $M$ is bounded below, then by \cite{w00} Theorem 7.1,
there exists $c_1, c_2>0$ such that (\ref{SP}) holds for $T=
c_1\rho$ and $\beta(s)= c_2s^{-d/2}.$ For simplicity, throughout
this section we assume that

\ \newline $(H)$\ \emph{The injectivity radius of $M$ is positive.}
\smallskip

\subsection{Lyapunov condition $e^{aV}$}

 In this context, one may readily generalizes the result of Theorem (\ref{propcat}) for the
 first case (\ref{propcat}.1), with the euclidean distance replaced by the Riemannian one, assuming
 (\ref{condcat}.1),  (\ref{condcat}.2) and (\ref{condcat}.3).

 \begin{theorem}\label{propcatm}
Assume $(H)$ and that (\ref{condcat}.1), (\ref{condcat}.2), (\ref{condcat}.3) are satisfied.
 Suppose moreover that for large $\rho$, $|\nabla V|(x) \leq \gamma (V(x))$.  Then $\mu$ will
satisfy a (SPI) inequality with function $\beta$ given by
$$\beta(s)= C (1 + e^{\eta^{-1}(c/s)} \, \gamma^n(\eta^{-1}(c/s))).$$
\end{theorem}
The second point of Theorem \ref{propcat} is more delicate as it relies on finer conditions on the
manifold and the potential, it should however be possible to give mild additional assumptions
ensuring such a result (for instance the so called ``rolling ball condition''). Remark that it
extends to the manifold case Kusuoka-Stroock's result (giving life to Remark
(2.49) in their paper).

\subsection{Lyapunov condition $e^{a\rho^b}$.}

We suppose moreover here that $M$ is a Cartan-Hadamard manifold with lower bounded Ricci
curvature.

If we try to use $W=e^{a\rho^b}$ for $\rr\ge 1,$ since $\Delta \rr$ is
bounded above on $\{\rr\ge 1\}$ (see for example Th.0.4.10 in \cite{Wbook}),  {\bf (L)} holds for
$$\phi:= ab \rho^{b-2} \psi$$ with
$$\psi := \<\nn \rr^2, \nn V\>  - \left(c + ab \rho^{b}\right)$$ for some constant $c>0$ provided
$\psi$ is positive for large $\rr.$

We may then extend Lemma \ref{lemcurvature} in the manifold context.
\begin{lemma}\label{lemcurvaturem}
If \eqref{eqcurvature} holds, then $\rho\<\nn\rr, \nabla V \> \geq V -
V(o) + c_0 \, \rho^2/2$.
\end{lemma}
\begin{proof} For $x\in M$, let $\xi:\ [0,\rr(x)]\to M$ be the
minimal geodesic from $o$ to $x$. Let

$$g(t)= t\<\nn\rr,\nn V\>(\xi_t),\ \ \ \  t\ge 0.$$ We have

$$g'(t) = \<\nn \rr, \nn V\>(\xi_t) +t \Hess_V(\nn
\rr,\nn\rr)(\xi_t)\ge c_0 t +\ff{d V(\xi_t)}{d t}.$$ This implies
the desired assertion by integrating both sides on $[0, \rr(x)].$
\end{proof}
\smallskip

We may thus state
\begin{proposition}\label{propcurvaturem} Let $M$ be a Cartan-Hadamard
manifold with Ricci curvature bounded below. Let $V$ satisfy
\eqref{eqcurvature}. Then one can find positive constants $c,C$ such
that $\mu$ satisfies some (SPI) with function $\beta$ (given below
for $s$ small enough) in the following cases
\begin{enumerate}
\item[(\ref{propcurvaturem}.1)] \quad $c_0 \geq 0$,  $V(x) \geq c' \rr^b$ for $\rho$ large enough
some $c'>0$ and $b>1$, $$\beta(s)= C \, e^{c (1/s)^{\frac{b}{2\left((b-1)\wedge 1\right)}}} \, .$$
\item[(\ref{propcurvaturem}.2)] \quad $c_0 \geq 0$,  $d' \rho^{b'} \geq V(x) \geq c' \rho^b$ for
$\rho$ large enough some $d', c'>0$ and $b' \geq b > 1$, $$\beta(s)=
C \, e^{c (1/s)^{\frac{b'}{b'+b -2}}} \, .$$
\item[(\ref{propcurvaturem}.3)] \quad $c_0 \leq 0$, for $\rr$ large
enough, $V(x) \geq (\varepsilon - c_0/2)  \rho^2$ for some
$\varepsilon
>0$, and $$\beta(s)= C \, e^{c (1/s)} \, .$$
\item[(\ref{propcurvaturem}.4)]  \quad $c_0 \leq 0$, for $\rho$ large enough, $d' \rho^{b'} \geq
V(x) \geq c' \rho^b$ and $\beta$ as in (\ref{propcurvature}.2).
\end{enumerate}
\end{proposition}
The first point of this proposition specialized tot the case $c_0>0$ enables us to recover
\cite[Th.1.3]{W01} which extends Bakry-Emery criterion to lower bounded Ricci curvature manifold.
It then extends the result to various $F$-Sobolev.
\begin{proof}
The proof follows exactly the same line than in the flat case so that case 1 and case 2  follows
once it is noted that  since $\Hess_V\ge 0$ implies the convexity of the level sets $\bar V_r,$ we
know that some Nash inequality holds on $\bar V_r$ according to the discussion in the previous
subsection and the boundedness of these level sets ensured by our hypotheses on $V$.

Let us now consider Case 3. Since $b=2$, our hypothesis implies that
for $2a < \varepsilon$, $\phi \geq D V$. But \eqref{eqcurvature} and
$\Hess_{\rr^2}\ge 2$ on Cartan-Hadamard manifolds imply  that the
level sets of $x \mapsto H= V+c_0 \rho^2/2$ are convex. Since $V \geq
D \rho^2$, one has $c r \leq V \leq r$ on $\bar H_r$. We may thus
mimic case 1, just replacing $\bar V_r$ by $\bar H_r$. Case 4 is
similar to the previous one just improving the bound on $\phi$ as in
case 2.
\end{proof}

\begin{remark}
Remark that in full generality, according to \cite{W04} Theorem 1.2 and the recent paper \cite{CW07},
there always exists  $T\in C^\infty(M)$ such that $d\lambda:= e^{-T(x)}d x$ satisfies a
logarithmic Sobolev inequality hence (SPI) with
 $\beta(s)= e^{s^{-1}}.$ Of course for practical purposes, this very general fact is not
 completely useful since $T$ is unknown.
\end{remark}


\bigskip
\bibliographystyle{plain}

\end{document}